\documentclass[12pt]{article}
\usepackage{amssymb}
\usepackage{amsmath,amsthm}
\usepackage[latin1]{inputenc}
\usepackage{graphicx}
\usepackage{hyperref}
\usepackage{enumerate}
\usepackage{tikz}
\usepackage{tkz-graph}
\usepackage{mathrsfs}
\usepackage{verbatim}

\hypersetup{colorlinks=true, linkcolor=blue, citecolor=blue, urlcolor=blue}

\newtheorem{remark}{Remark}

\newtheorem{definition}[remark]{Definition}

\newtheorem{theorem}[remark]{Theorem}

\newtheorem{corollary}[remark]{Corollary}

\title{Lexicographic metric spaces: basic properties and the
metric dimension}

\author{Juan Alberto Rodr\'{\i}guez-Vel\'{a}zquez 
\\
\\
 {\small Universitat Rovira i Virgili
}\\
{\small Departament d'Enginyeria Inform\`atica i Matem\`atiques}\\
 {\small  Av. Pa\"{\i}sos Catalans 26,
43007 Tarragona, Spain.} \\{\small juanalberto.rodriguez\@@urv.cat}
 }

\begin{document}
\maketitle

\begin{abstract}
 In this article, we introduce the concept of lexicographic metric space and, after  discussing some basic properties of these metric spaces, such as completeness, boundedness, compactness and separability, we obtain a formula for the metric dimension of any lexicographic metric space.
\end{abstract}

{\it Keyword}:
 Lexicographic metric spaces; Gravitational metric spaces; Metric spaces; Lexicographic product; Metric dimension.



{\it AMS Subject Classification numbers}:  	54E35;	54E45;  	54E50


\section{Introduction}
A metric space $M=(X,d)$ is formed by a set $X$ of points and a distance $d$ defined in $X$.
In particular, the vertex set of any connected graph, equipped with the shortest path distance, is a metric space. This suggests approaching problems related to metric parameters in graphs as more general problems in the context of metric spaces. However, in some cases, a problem has been formulated within the theory of metric spaces and has become popular and extensively studied in graph theory. This is the case of the metric dimension;  a theory introduced by Blumenthal \cite{Blumenthal1953} in 1953   in the general context of metric spaces, that became popular more than twenty years later,  after being introduced by Slater \cite{Slater1975,Slater1988} in 1975, and by Harary and Melter
\cite{Harary1976} in 1976, in the particular context of graph theory. In other cases, like the Cartesian product of graphs, and the Cartesian product of metric spaces, the theory has been widely studied in both contexts. In this paper, we introduce the study of lexicographic metric spaces as a natural generalization of the theory of lexicographic product graphs. In particular, we define the concept of lexicographic metric space in such a way that the set of vertices of any lexicographic product graph, equipped with the distance induced by the graph, is a lexicographic metric space. We study  some basic properties of lexicographic metric spaces, such as completeness, boundedness, compactness and separability. Furthermore, we obtain a formula for the metric dimension of any lexicographic metric space.   As a consequence of the study, we derive some results on a kind of metric spaces which are closely related to the lexicographic metric spaces, and that we call gravitational metric spaces.

The remainder of the article is structured as follows. In section 2 we define the concepts of lexicographic metric space and gravitational metric space, and we also justify the terminology used. Section  3 is devoted to the study of basic properties of these metric spaces.
Finally,  the metric dimension is studied in Section 4. We assume that the reader is familiar with the basic concepts and terminology of metric spaces. If this is not the case, we suggest the textbooks \cite{MR0267052,MR2244658}.

\section{Lexicographic metric spaces and gravitational metric spaces}\label{sectPrelim}

There are metric spaces with metrics that may differ from the discrete metric\footnote{The discrete metric on $X$ is given by $ d(x,y) = 0$ if $x = y$ and $d(x, y) = 1$ otherwise.} yet generate the same topology. Such spaces are called \textit{discrete metric spaces} \cite{MR2244658}. In order to avoid misunderstandings regarding this concept, we will emphasize it in the following definition.

\begin{definition}{\rm \cite{MR2244658}}
A metric space $M$ is called a \textit{discrete metric space} if and only if all its subsets are open (and therefore closed) in $M$.
\end{definition}

For instance, the set $\mathbb{N}$ with its usual metric inherited from the one-dimensional Euclidean space $\mathbb{R}$ is a discrete metric space, every finite metric space is a discrete metric space  and, in particular, the vertex set of any graph  equipped with the shortest path distance is a discrete metric space.

Given a metric space $M=(X,d)$, and a point $x\in X$, we define the
\textit{nearness of $x$} to be $$\eta(x)=\inf \{d(x,y):\;   y\in X\setminus\{x\}\}.$$
We define the
\textit{nearness of the metric space} to be $$\eta(M)=\inf \{ \eta(x):\;                                               x\in X \}.$$
If a metric space $M=(X,d)$ is equipped with the discrete metric, then $\eta(x)=1$ for every $x\in X$.
  As an example of a discrete metric space with $\eta(M)= 0$ we take $M=(X,d)$ where $X=\{1/n: \; n\in \mathbb{N}\}$ and $d$ is the usual Euclidean metric inherited from $\mathbb{R}$. Notice that this discrete metric space is not complete.
  \begin{theorem}\label{ConsequenceNearness0}
  Let $M=(X,d)$ be a metric space.  If $\eta(M)> 0$, then $M$ is a complete discrete metric space.
  \end{theorem}

  \begin{proof}
If $\eta(M)> 0$, then  every singleton is an open set. Hence, since the union of any collection of open sets is itself an open set, from  $\eta(M)> 0$ we deduce that $M$ is a discrete metric space.

 Let $(x_n)_{n\in \mathbb{N}}$ be a Cauchy sequence in $M$. By definition
  of Cauchy sequence, for $\epsilon=\eta(M)>0$ there exists $N\in \mathbb{N}$ such that for any $n,m>N$ we have $d(x_n,x_m)<\epsilon=\eta(M)$. Now, if  $x_n\ne x_m$, for some $n,m>N$, then by definition of nearness we  have $\eta(M)\le d(x_n,x_m)<\eta(M)$, which is a contradiction. Thus, $(x_n)_{n\in \mathbb{N}}$ is constant from  some  $N\in \mathbb{N}$, which implies that $(x_n)_{n\in \mathbb{N}}$ converges in $M$. Therefore, $M$ is a complete metric space.\end{proof}

 In general, the converse of Theorem \ref{ConsequenceNearness0} is not true.
 For instance, the metric space $M=(U,d)$, where
 $$U=\left\{2+\frac{1}{2},2+\frac{1}{2-1},3+\frac{1}{3},3+\frac{1}{3-1},\dots, n+\frac{1}{n},n+\frac{1}{n-1}, \dots\right\}$$ and $d$ is the metric inherited from the one-dimensional Euclidean space, is a complete discrete metric space, as any Cauchy sequence is constant from a given $N\in \mathbb{N}$. On the other hand, since $$d\left(n+\frac{1}{n},n+\frac{1}{n-1}\right)=\left|\left(n+\frac{1}{n}\right)-\left(n+\frac{1}{n-1}\right)\right|=\frac{1}{n(n-1)}\longrightarrow 0,
$$
we have that
 $\eta(M)=0$.

Notice that the vertex set $V$ of any  connected graph $G=(V,E)$ equipped with the shortest path distance is a discrete metric space with nearness equal to one.

 \begin{definition}
Let $M=(X,d_X)$ and $M'=(Y,d_Y)$ be two metric spaces such that $\eta(M)>0$. A  map $\rho: X\times Y\longrightarrow \mathbb{R}$  defined by
$$\rho((x,y),(x',y'))=\left\{\begin{array}{ll}

                            d_X(x,x'), & \text{ if } x\ne x', \\
                            \\
                          \displaystyle \min\{2\eta(x), d_Y(y,y')\}, & \text{ if }  x=x'.
                            \end{array} \right.
$$ is called a lexicographic distance, and the metric space $M\circ M'=(X\times Y, \rho)$ is called a lexicographic metric space.
\end{definition}

By a case study we can check that the map $\rho$ is a distance.
In order to provide a formula for  the diameter\footnote{The diameter of $M=(X,d)$ is defined to be $D(M)=\sup\{d(x,y):\; x,y\in X\}$.} of $M\circ M'$, we need to introduce the following parameter that we call the \emph{slack of a metric space} $M$:
$$\zeta(M)=\sup \{ \eta(x):\;  x\in X \}.$$
As a direct consequence of the definition of lexicographic distance,  the diameter of $M\circ M'$ is given by
$$D(M\circ M')=\max\{D(M), \min\{2\zeta(M),D(M')\}\}.$$

The name lexicographic metric space is inherited from graph theory.   Let us briefly recall the notion of the lexicographic product of two graphs.  The   lexicographic product $G \circ H$ of two graphs
$G$ and $H$
is the graph  whose vertex set is  $V(G \circ H)=  V(G)  \times V(H )$ and $(u,v)(x,y) \in E(G \circ H)$ if and only if $ux \in E(G)$ or $u=x$ and $vy \in E(H)$. For basic properties of the lexicographic product of two graphs we suggest the  books  \cite{Hammack2011,Imrich2000}. We denote by $d_G$ the shortest path distance of any graph $G$. It is well known that the shortest path distance in $G\circ H$ is expressed in terms of the shortest path distances in $G$ and $H$:
$$d_{G\circ H}((x,y),(x',y'))=\left\{\begin{array}{ll}

                            d_G(x,x'), & \text{ if } x\ne x', \\
                            \\
                          \displaystyle  \min\{2, d_H
                          (y,y')\}, & \text{ if }  x=x'.
                            \end{array} \right.
$$
Therefore, the set of vertices of any lexicographic product graph equipped with the shortest path distance is a lexicographic metric space, thus justifying the terminology used.

Notice that for any $x\in X$  the subspace $M'_x=(Y_x,d^x_Y)$ of $M\circ M'$ induced by the set $Y_x=\{x\}\times Y$ is equipped with the distance $d^x_Y((x,y),(x,y'))=\min\{2\eta(x),d_Y(y,y')\}$.  In fact, the metric space $M'_x$ is isometric to a metric space obtained from $M'$, which can be seen as a deformation of $M'$ where every point  $y\in Y$ attracts (and is attracted by)  every point
that is outside  the closed ball of center $y$ and radius $2\eta(x)$.  All these points remain  at distance  $2\eta(x)$ from $y$ in the deformed space. In particular, if the diameter of $M'$ satisfies $D(M')\le 2\eta(x)$, then $M'_x$ is isometric to $M'$. The metric space $M'_x$ locally, near every point, looks like patches of the space $M'$, but the global topology can be quite different.
From now on, we will refer to $M'_x$ as a \emph{gravitational space} of $M'=(Y,d_Y)$ with gravitation constant $2\eta(x)$. In Section 4 we will show that  gravitational metric spaces play an important role in the study of the metric dimension of lexicographic metric spaces.

Since a gravitational metric space of $M'=(Y,d_Y)$ only depends on its gravitational constant and on the structure of $M'$, it has his own identity, regardless of the definition of lexicographic metric space. For this reason, if there is no ambiguity, we will sometimes denote it by $M'_t=(Y,d_t)$, where $t>0$ is a constant and
\begin{equation}\label{GravitationalMetric}
d_t(y_i,y_j)=\min\{2t,d_Y(y_i,y_j)\}\text{ for every } y_i,y_j\in Y.
\end{equation}
 In fact, this metric has been used in the context of graph theory to study the metric dimension of lexicographic product graphs and corona product graphs \cite{Estrada-Moreno2014a,Estrada-Moreno2014b,RV-F-2013}. Moreover, a general study of the metric dimension of any graph equipped with this distance was proposed in \cite{AlejandroEstrada-MorenoRodriguez-Velazquez2016} and a general study in the context of metric spaces was initiated in \cite{Beardon-RodriguezVelazquez2016}.

\section{Completeness, boundedness, compactness and separability of lexicographic metric spaces and gravitational metric spaces}\label{SectionPropertiesLexicMetricSpaces}

To begin this section we proceed to state a criteria  for completeness of any gravitational metric space.

\begin{theorem}\label{ThCompletnessGravitationalSpace}
Let $t>0$ be a real number. A metric space $M=(X,d)$ is complete if and only if the gravitational metric space $M_t=(X, d_t)$ is complete.
\end{theorem}

\begin{proof}
The result mainly follows from the fact that for any $\epsilon\in (0,2t)$ and any $x\in X$, the  ball of center $x$ and radius $\epsilon$ in $M$ coincides with the  ball of center $x$ and radius $\epsilon$ in $M_t$.

We proceed to give the detail of the proof in one direction. Assume that $M_t$ is a complete metric space. Let $(x_n)_{n\in \mathbb{N}}$ be a Cauchy sequence in $M$. Hence, for any $\epsilon\in (0,2t)$, there exists $N\in \mathbb{N}$ such that  $d(x_n,x_m)<\epsilon$ for every $n,m>N$. Now, since $d(x_n,x_m)<\epsilon<2t$, we have that $d_t(x_n,x_m)=d(x_n,x_m)<\epsilon$, which  implies that $(x_n)_{n\in \mathbb{N}}$ is a Cauchy sequence in $M_t$. Therefore, since $(x_n)_{n\in \mathbb{N}}$ converges in $M_t$, it also converges in  $M$.

The proof in the other direction is quite similar. We omit the details.
\end{proof}

As we will show in the next result, the  completeness of $M\circ M'$ depends on the completeness of $M'$.

\begin{theorem}\label{Th-Completness}
A lexicographic metric space $M\circ M'$ is complete if and only if $M'$ is complete.
\end{theorem}

\begin{proof}We first assume that $M'=(Y,d_Y)$ is a complete metric space. Let $(z_n)_{n\in \mathbb{N}}$ be a Cauchy sequence in $M\circ M'$, where $z_n=(x_n,y_n)$ for every $n\in  \mathbb{N}$. By definition of Cauchy sequence, for every $\epsilon\in (0,\eta(M))$ there exists $N\in \mathbb{N}$ such that $\rho(z_n,z_m)< \epsilon$ for every $n,m>N$. Now, by definition of nearness,  if $x_n\ne x_m$, then  $\rho(z_n,z_m)=d_X(x_n,x_m)\ge \eta(M)>\epsilon$, which is a contradiction. Thus, $(x_n)_{n\in \mathbb{N}}$ is constant from $N$ and $\epsilon>\rho(z_n,z_m)=\min\{2\eta(x_n),d_Y(y_n,y_m)\}=d_Y(y_n,y_m)$, for every $n,m>N$.  As a result, $(y_n)_{n\in \mathbb{N}}$ is a Cauchy sequence
which converges in $M'$ if and only if $(z_n)_{n\in \mathbb{N}}$ converges in $M\circ M'$.
Therefore, from the completeness of the space  $M'$ we deduce the completeness of  $M\circ M'$.

Now, assume that $M\circ M'$ is a complete metric space. By a procedure similar to that developed above, by  fixing $x\in X$ and taking $\epsilon\in (0,\eta(x))$,  we deduce that for any Cauchy sequence  $(y_n)_{n\in \mathbb{N}}$ in $M'$,  the sequence $((x,y_n))_{n\in \mathbb{N}}$ is a Cauchy sequence in $M\circ M'$ which converges if and only if $(y_n)_{n\in \mathbb{N}}$ converges in $M'$. Therefore, from the completeness of the space $M\circ M'$ we deduce the completeness of $M'$.
\end{proof}

The concept of boundedness plays an important role in the theory of metric spaces. Whereas a metric space is \emph{bounded} if it is included in a single ball, a metric space $M=(X,d)$ is \emph{totally bounded} if for every $\epsilon>0$ there exists a finite set $S\subseteq X$ such that
$X=\cup_{s\in S} B_{\epsilon}(s)$, where $$B_{\epsilon}(s)=\{x\in X:\; d(x,s)<\epsilon\}$$ is the ball of center $s$ and radius $\epsilon$.
In the proof of Theorems~\ref{LemmaGravitationalTotallyBounded} and \ref{Th-TotallyBounded} we use the simple fact that  for positive numbers $\epsilon <\epsilon'$ and $S\subseteq  X$, if $X=\cup_{s\in S} B_{\epsilon}(s)$, then  $X=\cup_{s\in S} B_{\epsilon'}(s)$.

\begin{theorem}\label{LemmaGravitationalTotallyBounded}
Let $t>0$ be a real number. A metric space $M=(X,d)$ is totally bounded if and only if the gravitational metric space $M_t=(X, d_t)$ is totally bounded.
\end{theorem}

\begin{proof}
In this proof, balls in $M$ are denoted by $B$, while balls in $M_t$ are denoted by $B^t$.
The result immediately follows from the fact that $B_{\epsilon}(x)=B^t_{\epsilon}(x)$ for every $x\in X$ and $0<\epsilon<2t$.
\end{proof}

 \begin{theorem}\label{Th-TotallyBounded}
Let $M=(X,d)$ be a metric spaces with $\eta(M)>0$, and let $M'$ be a metric space. The lexicographic metric space $M\circ M'$ is totally bounded if and only if $|X|<+\infty$ and $M'$ is totally bounded.
\end{theorem}
 \begin{proof}

Let $M'=(Y,d_Y)$,  $0<\epsilon<\eta(M)/2$ and $S\subseteq X\times Y$ such that $X\times Y=\cup_{s\in S}B_{\epsilon}(s)$. Since for any pair of different points $x_1,x_2\in X$ and  $y_1,y_2\in Y$, we have that $B_{\epsilon}((x_1,y_1))\cap B_{\epsilon}((x_2,y_2))=\emptyset$, we can conclude that $|S|<+\infty$ if and only if (a) $|X|<+\infty$ and (b) for each $x\in X$ there exist $S_x=\{x\}\times S' \subseteq S$  such that $|S_x|<+\infty$ and $\{x\}\times Y=\cup_{s\in S_x}B_{\epsilon}(s)$. Notice that (b) is equivalent to say that  the  gravitational metric space $M'_x$ is totally bounded for every $x\in X$.  Therefore, by Theorem \ref{LemmaGravitationalTotallyBounded} we conclude the proof.
\end{proof}

We now proceed to study the compactness of lexicographic metric spaces.
The following known theorem, which is a characterization of compact metric spaces, is  as a generalization of the Heine-Borel theo\-rem.

\begin{theorem}{\rm \cite{MR0267052}}\label{Heine-Borel}
A metric space is compact if and only if it is complete and totally bounded.
\end{theorem}

 From Theorems \ref{ThCompletnessGravitationalSpace}, \ref{LemmaGravitationalTotallyBounded}   and \ref{Heine-Borel} we deduce the following result.

 \begin{theorem}\label{ThGravitationalCompact}
Let $t>0$ be a real number. A metric space $M=(X,d)$ is compact if and only if the gravitational metric space $M_t=(X, d_t)$ is compact.
\end{theorem}

 Our next result is a direct consequence of Theorems \ref{Th-Completness},  \ref{Th-TotallyBounded} and \ref{Heine-Borel}.
\begin{theorem}\label{CompactLexicographicMetricSpaces}
Let $M=(X,d)$ be a metric spaces with $\eta(M)>0$, and let $M'$ be a metric space. Then the lexicographic metric space $M\circ M'$ is compact if and only if $|X|<+\infty$ and $M'$ is compact.
\end{theorem}

A set $S$ of points of a metric space $M$ is dense in $M$ if every open ball of $M$ contains at least a point belonging to $S$. A metric space $M$ is separable if there exists a countable set of points which is dense in $M$.
By the same arguments used to get Theorem \ref{LemmaGravitationalTotallyBounded} we deduce the following result.

\begin{theorem}\label{LemmaGravitationalSeparableness}
Let $t>0$ be a real number. A metric space $M=(X,d)$ is separable if and only if the gravitational metric space $M_t=(X, d_t)$ is separable.
\end{theorem}

 \begin{theorem}\label{Th-Separable}
Let $M=(X,d)$ be a metric spaces with $\eta(M)>0$, and let $M'$ be a metric space. The lexicographic metric space $M\circ M'$ is separable if and only if $X$ is a countable set and  $M'$ is separable.
\end{theorem}
 \begin{proof}
We proceed to show that $M\circ M'$ is separable if and only if $X$ is a countable set and for any $x\in X$ the gravitational metric space $M'_{\eta(x)}$ is separable. After that, the proof is concluded by Theorem \ref{LemmaGravitationalSeparableness}.

We first assume that for any $x\in X$ the gravitational metric space $M'_{\eta(x)}$ is separable and $X$ is a countable set. For any $x\in X$, let $S_x \subseteq \{x\}\times Y$ be a countable set which is dense  in $M'_{\eta(x)}$. Notice that  $S=\cup_{x\in X}S_x$ is a countable set, as the countable union of countable sets is countable. Now, a ball $B_{\epsilon}((x,y))$ in  $M'_{\eta(x)}$ will be distinguished from a ball ${\bf B}_{\epsilon}((x,y))$ in $M\circ M'$ by the bold type used.
Since  $S_x$ is dense in $M'_{\eta(x)}$, we can conclude that for any $\epsilon>0$ and $y\in Y$,
$${\bf B}_{\epsilon}((x,y))\cap S\supseteq {B}_{\epsilon}((x,y))\cap S_x\ne \emptyset.$$
Hence, $S$ is dense in $M\circ M'$, and so $M\circ M'$ is separable.

Conversely,  assume that $M\circ M'$ is separable. Since every subspace of a separable metric space is separable, we can conclude that $M'_{\eta(x)}$ is separable for every $x\in X$. Finally, if $U$ is a  countable  set which is dense in $M\circ M'$  and   $x\in X$, then $\{x\}\times Y\supseteq{\bf B}_{\eta(M)/2}((x,y))\cap U \ne \emptyset$, which implies that $|X|\le |U|$. Thus, $X$ is a countable set.
 \end{proof}

\section{The metric dimension}

The metric dimension of a general metric space was introduced for the first time by Blumenthal \cite{Blumenthal1953} in 1953. This theory
 attracted little attention until, about twenty years later, it was applied to the
distances between vertices of a graph
\cite{Harary1976,Slater1975,Slater1988}.
 Since then it has been frequently used in graph
theory, chemistry, biology, robotics and many other disciplines.
More recently, in \cite{Sheng2013,Beardon-RodriguezVelazquez2016,MR3276739}, the
theory of metric dimension was developed further for general metric spaces.
Here we develop the idea of the metric dimension  in lexicographic metric spaces. As the theory is trivial when the space
has one point, we shall assume that any space we are
considering has \emph{at least two points}.

Let $M=(X,d)$ be a metric space. If $X$ is a finite set, we denote its
cardinality by $|X|$; if $X$ is an infinite set, we put $|X| =
+\infty$. In fact, it is possible to develop the theory with $|X|$ any
cardinal number, but we shall not do this. The distances from a point
$x$ in $X$ to the points $a$ in a subset $A$ of $X$ are given by the
function $a \mapsto d(x,a)$, and the subset $A$ is said to
\emph{resolve} $M$ if each point $x$ is uniquely determined by this
function.  Thus $A$ resolves $M$ if and only if $d(x,a)=d(y,a)$
for all $a$ in $A$ implies that $x=y$; informally, if an object in
$M$ knows its distance from each point of $A$, then it knows exactly
where it is located in $M$. The class $\mathcal{R}(M)$ of subsets of
$X$ that resolve $M$ is non-empty since $X$ resolves $M$. The
\emph{metric dimension} $\dim(M)$ of $M=(X,d)$ is defined as
$$\dim(M)=\min\{|S|: \; S\in \mathcal{R}(M).\}$$
The sets in $\mathcal{R}(M)$ are
called the \emph{metric generators}, or \emph{resolving subsets}, of
$M$, and $S$ is a \emph{metric basis} of $M$ if $S\in \mathcal{R}(M)$
and $|S|= {\rm dim}(M)$.

This terminology comes from the fact that a metric generator of a metric space $M=(X,d)$
induces a \emph{global co-ordinate system} on $M$. For example,
if $(x_1,\ldots,x_r)$ is an ordered metric generator of $M$, then the
map $\psi:X \to \mathbb{R}^r$ given by
\begin{equation}\label{151123a}
\psi(x)= \Big(d(x,x_1), \ldots,d(x,x_r)\Big)
\end{equation}
is injective (for this vector determines $x$), so that $\psi$ is a
bijection from $X$ to a subset of $\mathbb{R}^r$, and the metric space  inherits its
co-ordinates from this subset. As the following result shows, a stronger conclusion arises when $M$ is a compact metric space.

\begin{theorem}{\rm \cite{Beardon-RodriguezVelazquez2016}}\label{151123d}
If $M$ is a compact metric space with ${\rm dim}(M)=r <
+\infty$, then $M$ is homeomorphic to a compact subspace of the Euclidean space
$\mathbb{R}^r$.
\end{theorem}

As the following result shows, the boundedness of $M$ affect directly the metric dimension of its gravitational metric spaces.

\begin{theorem}{\rm\cite{AlejandroEstrada-MorenoRodriguez-Velazquez2016}}\label{Dim^tUnboundedSpaces}
Let $M=(X,d)$ be a metric space, and $t> 0$ a real number.  If $M$ is unbounded,  then the gravitational metric space $M_t=(X,d_t)$ has metric dimension $\dim(M_t)=+\infty$.
\end{theorem}

In general, the converse of Theorem~\ref{Dim^tUnboundedSpaces} does not hold. For instance, if $M=(\mathbb{Z},d)$, where $d$ is the discrete metric, then $M$ is bounded and $\dim(M_t)=+\infty$ for every $t>0$.  On the other hand, by Theorem~\ref{Dim^tUnboundedSpaces}, we can claim that if $M$ is a metric space such that $\dim(M_t)<+\infty$, then $M$ is bounded,
 but we can not claim that $M$ is totally bounded, as there are bounded metric spaces that are not totally bounded. Our next result
 provides a sufficient condition for a metric space to be totally bounded.

\begin{theorem}
Let $M$ be a metric space. If $\dim(M_{\epsilon})<+\infty$ for all $\epsilon >0$, then $M$ is totally bounded.
\end{theorem}

\begin{proof}
Let $S\subseteq X$ be a finite metric generator of $M_{\epsilon}=(X,d_{\epsilon})$. Let $S_1=X\setminus \left(\cup_{x\in S}B_{2\epsilon}(x)\right)$, where $B_{2\epsilon}(x)$ is the ball of center $x$ and radius $2\epsilon$ in $M=(X,d)$. Since $S$ is a metric generator of $M_{\epsilon}$, we have that $|S_1|\le 1$. Therefore, since
$X=\cup_{x\in S\cup S_1}B_{2\epsilon}(x)$ and $|S\cup S_1|$  is a finite set, the result follows.
\end{proof}

We now proceed to study the metric dimension of lexicographic metric spaces. To this end, we need to introduce some additional terminology.
We say that two points $a,b\in X$ are \emph{twins} in the metric space $M=(X,d_X)$ if and only if $d_X(a,c)=d_X(b,c)$ for every $c\in X\setminus \{a,b\}$. We define the \emph{twin equivalence relation} $\equiv$ on $X$ as follows:
$$x \equiv y \leftrightarrow x \text{ and } y \text{ are twins.} $$
Let $\dot{x}\subseteq X$ be a non-singleton  twin equivalence class. Notice that the distance between the elements in $\dot{x}$ is constant. So, let $\ell_x=d_X(x_i,x_j)$, for every $x_i,x_j\in \dot{x}$. Given a metric space $M'=(Y,d_Y)$, let $X_Y$ be the set of non-singleton twin equivalence classes of $M$ such that for any $x\in \dot{x}\in X_Y$ and any  metric basis $S_x$ of the gravitational space $M'_x$ there exists $z\in  Y_x=\{x\}\times Y$ such that $d_{\eta(x)}(z,s)=\ell_x$ for every $s \in S_x$. By definition of metric basis, once fixed $S_x$, if $z$ exists, then it is unique.
With this notation in mind we can state the following result.

\begin{theorem}\label{MainTheoremMD-LexicMetricSpaces}
 If $M=(X,d_X)$ is a metric space with $\eta(M)>0$, then for any metric space $M'=(Y,d_Y)$,
$$\dim(M\circ M')= \sum_{x\in X}\dim(M'_x)+\sum_{\dot{x}\in X_Y}(|\dot{x}|-1).$$
\end{theorem}

 \begin{proof}
The result is trivially true if $\dim(M\circ M')=+\infty$, so we can assume that there exists a finite metric basis, say  $W$, of $M\circ M'$. Let $x\in X$ and $y_1,y_2\in Y$, where $y_1\ne y_2$. Since for any $a\in X\setminus \{x\}$ and $b\in Y$,
$$\rho((x,y_1),(a,b))= d_X(x,a)=\rho((x,y_2),(a,b)),$$
we can conclude that $W_x=W\cap Y_x$ is a metric generator of the gravitational metric space $M'_x$.
Hence, $|W_x|\ge \dim(M'_x)$ for every $x\in X$. Now, let $\dot{x}\in X_Y$ and $x_i,x_j\in \dot{x}$. If $|W_{x_i}|= \dim(M'_{x_i})$ and $|W_{x_j}|= \dim(M'_{x_j})$, then $W_{x_i}$ is a metric basis of $M'_{x_i}$ and $W_{x_j}$ is a metric basis of $M'_{x_j}$, which implies that there exist
$w_i\in  Y_{x_i}$ and $w_j\in  Y_{x_j}$ such that $d_{\eta(x_i)}(w_i,w)=\ell_{x_i}=d_X(x_i,x_j)=\ell_{x_j}=d_{\eta(x_j)}(w_j,w')$ for every $w \in W_{x_i}$ and $w' \in W_{x_j}$. Hence, $w_i$ and $w_j$ are not distinguished by the elements in $W$, which is a contradiction. Thus, $|W_{x_i}|\ge \dim(M'_{x_i})+1$ or $|W_{x_j}|\ge \dim(M'_{x_j})+1$.
Therefore,
$$\dim(M\circ M')=|W|=\sum_{x\in X}|W'_x|\ge \sum_{x\in X}\dim(M'_x)+\sum_{\dot{x}\in X_Y}(|\dot{x}|-1).$$
It remains to show that $\dim(M\circ M')\le \sum_{x\in X}\dim(M'_x)+\sum_{\dot{x}\in X_Y}(|\dot{x}|-1).$
To this end, for every $ \dot{x}\in X_Y$, and for all but one $x\in \dot{x}$, we fix a metric basis $S_x$ of $M'_x$ and define $S'_x=S_x\cup \{z\}$, where $z\in Y_x$ is the only point in $Y_x$ such that $d_{\eta(x)}(z,s)=\ell_x$ for every $s\in S_x$. Now, for the remaining $x\in X$ we take $S'_x$ as a metric basis of $M'_x$ such that if $x\in \dot{x}\not \in X_Y$, then for every $w\in Y_x$ there exist $s_i,s_j\in S'_{x}$ such that $d_{\eta(x)}(w,s_i)\ne \ell_{x}$ or $d_{\eta(x)}(w,s_j)\ne \ell_x$.  We claim that $S=\cup_{x\in X}S'_x$ is a metric generator of $M\circ M'$. To see this, we differentiate three cases for two different points $u=(u_1,u_2)$ and $v=(v_1,v_2)$  of $M\circ M'$ not belonging to $S$.
\\
\noindent Case 1. $u_1=v_1$. In this case,  there exists a point in $S'_{u_1}\subseteq S$ which distinguishes $u$ and $v$, as $S'_x$ is a metric generator of $M'_x$  for any $x\in X$.
\\
\noindent Case 2. $u_1\ne v_1$ and $v_1\in \dot{u_1}$. Since $u_1$ and $v_1$ are twins,  by the way in which we have chosen $S'_{u_1}$ and $S'_{v_1}$ we can claim that at least one of the following conditions holds:
\\
\noindent (a) There exist $s_i,s_j\in S'_{u_1}$ such that $d_{\eta(u_1)}(u,s_i)\ne \ell_{u_1} $ or $ d_{\eta(u_1)}(u,s_j)\ne \ell_{u_1}$.
\\
(b) There exist $s_i,s_j\in S'_{v_1}$ such that $d_{\eta(v_1)}(v,s_i)\ne \ell_{v_1} $ or $d_{\eta(v_1)}(v,s_j)\ne \ell_{v_1}$.
\\
From each one of these conditions we deduce that $\rho(s_i,u)\ne \rho(s_i,v)$ or $\rho(s_j,u)\ne \rho(s_j,v)$. For instance, suppose that (a) holds. In such a case, $\rho(s_i,u)=d_{\eta(u_1)}(u,s_i)$, $\rho(s_j,u)=d_{\eta(u_1)}(u,s_j)$ and $\rho(s_i,v)=\rho(s_j,v)=d_X(u_1,v_1)=\ell_{u_1}=\ell_{v_1}$. Thus, if $\rho(s_i,u)= \rho(s_i,v)$, then  $\rho(s_j,u)\ne \rho(s_j,v)$, and if $\rho(s_j,u)= \rho(s_j,v)$, then  $\rho(s_i,u)\ne \rho(s_i,v)$.
\\
\noindent Case 3.  $v_1\not\in \dot{u_1}$. Since $u_1$ and $v_1$ are not twins,  there exists $x\in X\setminus \{u_1,v_1\}$ such that $d_X(x,u_1)\ne  d_X(x,v_1)$. Hence, for every $s\in S'_x$ we have that $\rho(s,u)=d_X(x,u_1)\ne  d_X(x,v_1)=\rho(s,v).$

According to the three cases above, $S$ is a metric generator of $M\circ M'$ and so
$$\dim(M\circ M')\le |S|=\sum_{x\in X}|S'_x|=\sum_{x\in X}\dim(M'_x)+\sum_{\dot{x}\in X_Y}(|\dot{x}|-1),$$
as required.\end{proof}

The metric dimension of lexicographic product graphs was previously studied in \cite{JanOmo2012,Saputro2013}.
As we can expect, we can apply the result above to the particular case of graphs. When we discuss a connected graph $G=(V,E)$, we consider the metric space $(V,d)$, where $V$ is the vertex set of
$G$, and $d$ is the shortest path metric in which the distance between
two vertices is the smallest number of edges that connect them.
In fact,  as a particular case of Theorem \ref{MainTheoremMD-LexicMetricSpaces}, we obtain a formula for the metric dimension of any lexicographic product graph,  which is a way of expressing the three main results obtained in~\cite{JanOmo2012} in a unified manner.  Obviously,  we  can also apply the result above to study the metric dimension of
 weighted graphs.

The following result is a direct consequence of Theorem \ref{MainTheoremMD-LexicMetricSpaces}.

\begin{corollary}
Let  $M=(X,d_X)$ and $M'=(Y,d_Y)$ be two metric spaces such that $\eta(M)>0$. If $|X|=+\infty$ or, if there exists $x\in X$ such that $\dim(M'_x)=+\infty$, then
$\dim(M\circ M')=+\infty.$
\end{corollary}

From Theorems \ref{Dim^tUnboundedSpaces}  and \ref{MainTheoremMD-LexicMetricSpaces} we deduce the following result.

\begin{theorem}
Let $M$ be a metric spaces  such that $\eta(M)>0$. If $M'$ is an unbounded metric space, then
$\dim(M\circ M')=+\infty.$
\end{theorem}

From the proof of Theorem \ref{MainTheoremMD-LexicMetricSpaces} we learned that the intersection of any metric basis of $M\circ M'$ with the set of points of any specific gravitational metric space $M'_x$ is a metric generator of $M'_x$. In this sense,  we can relate Theorems \ref {ThGravitationalCompact}, \ref{CompactLexicographicMetricSpaces} and \ref{151123d} as follows.

\begin{theorem}\label{CompactLexicographicMetricSpacesOfFiniteDim}
Let $M=(X,d_X)$  be a metric space  such that $\eta(M)>0$. If  $M'$ is a compact metric space and $\dim(M\circ M')=m<+\infty$, then the following assertions hold.
\begin{enumerate}[{\rm (i)}]
\item $M\circ M'$ is homeomorphic to a compact subset of the Euclidean space~$\mathbb{R}^m$.

\item $|X|<+\infty$ and for each $x_i\in X$ there exists a positive integer $m_i$ such that $\sum_{i=1}^{|X|}m_i\le m$ and the gravitational metric space $M'_{x_i}$ is homeomorphic to a compact subset of the Euclidean space
$\mathbb{R}^{m_i}$.
\end{enumerate}
\end{theorem}

 If  every  twin equivalence class of $M$ is singleton, then we say that $M$ is a \emph{twin-free} metric space. From Theorem \ref{MainTheoremMD-LexicMetricSpaces} we deduce the following result.

\begin{corollary}\label{MD-LexicMetricSpaces-TwinsFree}
 If $M=(X,d_X)$ is a twin-free metric space with $\eta(M)>0$, them for any metric space $M'$,
$$\dim(M\circ M')= \sum_{x\in X}\dim(M'_x).$$
\end{corollary}

If $M=(X,d_X)$ and $M'=(Y,d_Y)$ are two metric spaces such that $D(M')< \eta(M)$, then $X_Y=\emptyset$. Therefore, from Theorem \ref{MainTheoremMD-LexicMetricSpaces} we deduce the following result.

\begin{corollary}\label{RelatDiameterNearness}
Let $M=(X,d_X)$ and $M'=(Y,d_Y)$ be two metric spaces. If $D(M')< \eta(M)$, then
$$\dim(M\circ M')=|X|\dim(M').$$
\end{corollary}

By Theorem~\ref{Dim^tUnboundedSpaces} we learned that  the study of the metric dimension of $M \circ M' $ should be restricted to cases where the gravitational metric spaces are bounded. The next result  shows how  to obtain a metric space $M''$ from $M'$, which satisfies the premises of  Corollary~\ref{RelatDiameterNearness} and $\dim(M'')=\dim(M').$

\begin{theorem}\label{CorollaryFromNon-bounded-to-bounded}
Let $M=(X,d_X)$ be a metric space  such that $\eta(M)>0$, and let $M'=(Y,d_Y)$ be a (non-necessarily bounded) metric space. If  $M''=(Y,d^*)$, where $d^*=\frac{\eta(M) d_Y}{\eta(M)+ d_Y}$, then
$$\dim(M\circ M'')=|X|\dim(M')=|X|\dim(M'').$$
\end{theorem}

\begin{proof}
For any $y_1,y_2\in Y$ we have that $$d^*(y_1,y_2)=\frac{\eta(M) d_Y(y_1,y_2)}{\eta(M)+ d_Y(y_1,y_2)}<\eta(M).$$
Hence, $D(M'')<\eta(M)$ and from Corollary \ref{RelatDiameterNearness} we deduce that $\dim(M\circ M'')=|X|\dim(M'').$ Now, since $d^*(y,y_1)<d^*(y,y_2)$ if and only if $d_Y(y,y_1)<d_Y(y,y_2)$, we can conclude that $\dim(M')=\dim(M'')$. Therefore, the result follows.
\end{proof}

We can apply the result above in the field of graph theory as follows. Let $G=(V,E)$ be a simple and  connected graph, let $M=(V,d)$ be the metric space associated to $G$,  and $M'=(Y,d_Y)$  a (non-necessarily bounded) metric space.
We construct a bounded metric space $M''=(Y,d^*)$ equipped with the metric $d^*=\frac{d_Y}{1+d_Y}$.
Notice that in this case we have $D(M'')<1= \eta(M)$. By  Theorem~\ref{CorollaryFromNon-bounded-to-bounded}, we have that $\dim(M\circ M'')=|V|\dim(M'')=|V|\dim(M').$ Now, there exists a complete weighted graph $G^*$ whose vertex set is $Y$ and the weight of any edge $\{u,v\}$ equals $d^*(u,v)$.  Therefore, the metric dimension of the lexicographic product graph $G\circ G^*$ is $\dim(G\circ G^*)=|V|\dim(G^*)$.
For instance, if $M'$ is the $n$-dimensional Euclidean space, then we have that $\dim(G\circ G^*)=|V|(n+1)$, as $\dim(\mathbb{R}^n)=n+1$.


\section*{Acknowledgements}
The author would like to thank his colleagues Douglas J. Klein, Cong X. Kang and Eunjeong Yi, who invited him to do a research stay at Texas A$\&$M University in 2018. The results included in this paper were crafted there.  He also would thank Charles Hudson and Alejandro Estrada-Moreno for the review of the manuscript.
This research was supported in part by  the Spanish  government  under  the grants MTM2016-78227-C2-1-P and PRX17/00102.


\bibliographystyle{elsart-num-sort}



\end{document}